\documentclass[11pt,reqno]{amsart}
\usepackage[top=1.2in, bottom=1.2in, left=1.2in, right=1.2in]{geometry}
\usepackage{amsfonts, amssymb, amsmath, mathtools, dsfont}
\usepackage[linktoc=page, colorlinks, linkcolor=blue, citecolor=blue]{hyperref}
\usepackage{enumerate, commath}

\usepackage{graphicx}

\usepackage{subcaption} 
\usepackage[font=small]{caption} 
\numberwithin{equation}{section}

\DeclareMathOperator{\E}{\mathbb{E}}

\DeclareMathOperator{\Var}{Var}
\DeclareMathOperator{\codim}{codim}
\DeclareMathOperator{\dist}{dist}

\DeclareMathOperator{\tr}{tr}

\renewcommand{\Pr}[2][]{\mathbb{P}_{#1} \left\{ #2 \rule{0mm}{3mm}\right\}}
\newcommand{\ip}[2]{\langle#1,#2\rangle}

\def \N {\mathbb{N}}
\def \P {\mathbb{P}}
\def \R {\mathbb{R}}

\def \EE {\mathcal{E}}

\def \II {\mathcal{I}}

\def \a {\alpha}

\def \e {\varepsilon}
\def \d {\delta}
\def \l {\lambda}
\def \s {\sigma}

\def \tran {\mathsf{T}}

\def \HS {\mathsf{HS}}
\def \op {\mathsf{op}}
\def \Lip {\mathsf{Lip}}
\def \one {{\textbf 1}}

\newcommand{\red}{\textcolor{red}}

\newtheorem{theorem}{Theorem}[section]
\newtheorem{proposition}[theorem]{Proposition}
\newtheorem{corollary}[theorem]{Corollary}
\newtheorem{lemma}[theorem]{Lemma}

\newtheorem{definition}[theorem]{Definition}

\theoremstyle{remark}
\newtheorem{remark}[theorem]{Remark}

\begin{document}

\title{Concentration inequalities for random tensors}

\author{Roman Vershynin}
\date{\today}

\address{Department of Mathematics, University of California, Irvine}
\email{rvershyn@uci.edu}

\thanks{Work in part supported by the U.S. Air Force grant FA9550-18-1-0031.}

\begin{abstract}
 We show how to extend several basic concentration inequalities
 for simple random tensors
 $X = x_1 \otimes \cdots \otimes x_d$ where all $x_k$ are independent 
 random vectors in $\R^n$ with independent coefficients. 
 The new results have optimal dependence
 on the dimension $n$ and the degree $d$.
 As an application, we show that random tensors are well conditioned: 
 $(1-o(1)) n^d$ independent copies of the simple 
 random tensor $X \in \R^{n^d}$
 are far from being linearly dependent with high probability. 
 We prove this fact for any degree $d = o(\sqrt{n/\log n})$ and conjecture that it is true 
 for any $d = O(n)$.
\end{abstract}

\maketitle


\section{Introduction}


Concentration inequalities form a powerful toolset in probability theory and its many applications; 
see e.g. \cite{BLM, Ledoux, Ledoux-Talagrand}.
Perhaps the best known member of this large family of results 
is the {\em Gaussian concentration inequality}, which states that 
a standard normal random vector $x$ in $\R^n$ satisfies
\begin{equation}	\label{eq: conc Gaussian}
\Pr{ \abs{ f(x) - \E f(x) } > t } 
\le 2 \exp \Big( -\frac{t^2}{2\norm{f}_\Lip^2} \Big)
\end{equation}
for any Lipschitz function $f : (\R^n, \norm{\cdot}_2) \to \R$, 
see e.g. \cite[Theorem~5.6]{BLM}.

The Gaussian concentration inequality can be extended to some more general distributions 
on $\R^n$. A remarkable situation where this is possible is where $x$ has a product distribution 
with {\em bounded} coordinates and $f$ is {\em convex}. 
This result is due to by M.~Talagrand \cite{Talagrand}; see \cite[Section~4.2]{Ledoux}, 
\cite[Section~7.5]{BLM}:

\begin{theorem}[Convex concentration]		\label{thm: convex concentration}
  Let $f : (\R^n, \norm{\cdot}_2) \to \R$ be a convex and Lipschitz function. 
  Let $x$ be a random vector in $\R^n$ whose coordinates are 
  independent random variables that are bounded a.s. 
  Then, for every $t \ge 0$, we have
  \begin{equation}	\label{eq: Talagrand}
  \Pr{ \abs{ f(x) - \E f(x) } > t } \le 2 \exp \Big( -\frac{ct^2}{\norm{f}_\Lip^2} \Big).
  \end{equation}
  Here $c>0$ depends only on the bound on the coordinates.
\end{theorem}

The boundedness assumption in this result unfortunately  
excludes Gaussian distributions and many others. 
Significant efforts were made to extend Gaussian concentration
to more general, not necessarily bounded, distributions, 
see e.g. \cite{A, AS, GRST} and the references therein.
One such result, which 
holds for a general random vector $x$ with independent {\em subgaussian} coordinates, 
is for {\em Euclidean functions}, 
i.e. the functions of the form $f(x) = \norm{Ax}_H$ where $A$ is a linear operator from 
$\R^n$ into a Hilbert space. 

\begin{theorem}[Euclidean concentration]	\label{thm: Euclidean concentration}
  Let $H$ be a Hilbert space and 
  $A: (\R^n, \norm{\cdot}_2) \to H$ be a linear operator. 
  Let $x$ be a random vector in $\R^n$ whose coordinates are 
  independent, mean zero, unit variance, subgaussian random variables.
  Then, for every $t \ge 0$, we have
  \begin{equation}	\label{eq: conc subgaussian vectors}
  \Pr{ \abs{\norm{Ax}_H - \norm{A}_\HS} \ge t } 
  \le 2 \exp \Big( -\frac{ct^2}{\norm{A}_\op^2} \Big).
  \end{equation}
  Here $c>0$ depends only on the bound on the subgaussian norms. 
\end{theorem}

In this result, $\norm{A}_\HS$ and $\norm{A}_\op$
denote the Hilbert-Schmidt and operator norms of $A$, respectively.
Theorem~\ref{thm: Euclidean concentration} 
can be derived from Hanson-Wright concentration inequality for quadratic forms
\cite{RV}, see \cite[Section~6.3]{V book}.

\subsection{New results}

The goal of this paper is to demonstrate how
Theorems~\ref{thm: convex concentration} and \ref{thm: Euclidean concentration}
can be extended for {\em simple random tensors}. These are tensors of the form
$$
X \coloneqq x_1 \otimes \cdots \otimes x_d
$$
where $x_k$ are independent random vectors in $\R^n$  
whose coordinates are independent, mean zero, unit variance random variables 
that are either bounded a.s. (in Theorem~\ref{thm: convex concentration}) or subgaussian 
(in Theorem~\ref{thm: Euclidean concentration}). 

Can we expect that a concentration inequality like \eqref{eq: Talagrand} or \eqref{eq: conc subgaussian vectors} can hold for simple random tensors, i.e. for $f(X)$ where $f : (\R^{n^d}, \norm{\cdot}_2) \to \R$? Not really: if such inequality did hold, then it would imply that $\Var(f(X)) = O(1)$, but this is not the case. 
Indeed, consider the simplest case where $f$ is given 
by the Euclidean norm, i.e. $f(X) \coloneqq \norm{X}_2$, and let all $x_k$ be standard
normal random vectors in $\R^n$. Recall that a standard normal vector $x$ in $\R^n$ satisfies
$$
\E \norm{x}_2^2 = n
\quad \text{and} \quad
\E \norm{x}_2 \le \sqrt{n-c}
$$
for $n$ large enough, where $c>0$ is an absolute constant.\footnote{To check the second bound, write $\norm{x}_2^2 = n + \sum_{i=1}^n (x_i^2 - 1)$ and observe that, by the central limit theorem,
  the sum is approximately $\sqrt{2n} \, g$ where $g \sim N(0,1)$. Thus $\norm{x}_2 \approx \sqrt{n} + g/\sqrt{2}$, so $\Var(\norm{x}_2) \gtrsim c$. On the other hand, $\Var(\norm{x}_2) = \E \norm{x}^2 - (\E \norm{x})^2 = n - (\E \norm{x})^2$. Thus, $\E \norm{x} \le \sqrt{n-c}$.}
Then 
\begin{align*}
  \Var(f(X)) 
    &= \E \norm{X}_2^2 - ( \E \norm{X}_2 )^2 
    = \big( \E \norm{x}_2^2 \big)^d - \big( \E \norm{x}_2 \big)^{2d} \\
    &\ge n^d - (n-c)^d 
    \ge c d (n-c)^{d-1}
    	\quad \text{(by the binomial expansion)} \\
    &\asymp dn^{d-1}.
  \end{align*}
Thus, the strongest concentration inequality we can hope for must have the form
\begin{equation}	\label{eq: conc tensors conj}
\Pr{ \abs{ f(X) - \E f(X) } > t } 
\le 2 \exp \Big( -\frac{ct^2}{d n^{d-1} \norm{f}_\Lip^2} \Big).
\end{equation}

Such inequality, however, can not be true for large $t$. 
The coefficients of $X$ are $d$-fold products of subgaussian random variables, 
and such products for $d \ge 2$ typically have tails that are heavier than subgaussian. 
Nevertheless, we may still hope that the inequality \eqref{eq: conc tensors conj} might hold 
for all $t$ in some interesting range, for example for $0 \le t \le \abs{\E f(X)}$. 
This is what we prove in the current paper. 

\begin{theorem}[Convex concentration for random tensors]		\label{thm: convex concentration tensor}
  Let $n$ and $d$ be positive integers and 
  $f : (\R^{n^d}, \norm{\cdot}_2) \to \R$ be a convex and Lipschitz function. 
  Consider a simple random tensor $X \coloneqq x_1 \otimes \cdots \otimes x_d$ 
  in $\R^{n^d}$, where all $x_k$ are independent random vectors in $\R^n$ whose coordinates
  are independent, mean zero, unit variance random variables that are bounded a.s.
  Then, for every $0 \le t \le 2 \big( \E \abs[0]{f(X)\red{-f(0)}}^2 \big)^{1/2}$, we have\footnote{The published version of this paper is missing the ``$-f(0)$'' part; I  apologize about this typo. Actually, one can make the conclusion of Theorem~\ref{thm: convex concentration tensor} hold for any $t \ge 0$ if we add the term $2\exp(-cn/d)$ in the right hand side. This is obvious from the proof, see page \pageref{page: using f assumption}.}
  
  $$
  \Pr{ \abs{ f(X) - \E f(X) } > t } 
  \le 2 \exp \Big( -\frac{ct^2}{d n^{d-1} \norm{f}_\Lip^2} \Big).
  $$
  Here $c>0$ depends only on the bound on the coordinates.
\end{theorem}

\begin{theorem}[Euclidean concentration for random tensors]	\label{thm: Euclidean concentration tensor}
  Let $n$ and $d$ be positive integers, $H$ be a Hilbert space and 
  $A: (\R^{n^d}, \norm{\cdot}_2) \to H$ be a linear operator. 
  Consider a simple random tensor $X \coloneqq x_1 \otimes \cdots \otimes x_d$ 
  in $\R^{n^d}$, where all $x_k$ are independent random vectors in $\R^n$ whose coordinates
  are independent, mean zero, unit variance, subgaussian random variables.
  Then, for every $0 \le t \le 2 \norm{A}_\HS$, we have
  $$
  \Pr{ \abs{\norm{AX}_H - \norm{A}_\HS} \ge t } 
  \le 2 \exp \Big( -\frac{ct^2}{d n^{d-1} \norm{A}_\op^2} \Big).
  $$
  Here $c>0$ depends only on the bound on the subgaussian norms. 
\end{theorem}

\begin{remark}[The range of concentration inequalities]
  Our argument shows that the main results actually hold in a somewhat wider range of $t$, namely for 
  $0 \le t \le 2n^{d/2} \norm{f}_\Lip$ in Theorem~\ref{thm: convex concentration tensor} and 
  $0 \le t \le 2n^{d/2} \norm{A}_\op$ in Theorem~\ref{thm: Euclidean concentration tensor}.
\end{remark}

\subsection{Related results}
  Several existing techniques are already known to shed light on tensor concentration.
  Indeed, the quantity $\norm{AX}_H^2$ in Theorem~\ref{thm: Euclidean concentration tensor} 
  can be expressed a polynomial of degree $2d$ in $nd$ independent subgaussian 
  random variables, which are the coefficients of the random vectors $x_i$.
  A remarkable result of Latala \cite{Latala} provides two-sided bounds on the moments
  of polynomials of independent normal random variables, 
  or {\em Gaussian chaoses}.
  These moment bounds can be translated
  into the following concentration inequality that is valid for any fixed tensor 
  $B \in \R^{n^d}$ and a simple random tensor 
  $X \coloneqq x_1 \otimes \cdots \otimes x_d$ composed of 
  independent, mean zero, standard normal random vectors $x_k$ in $\R^n$:
  \begin{equation}	\label{eq: Latala}
  \Pr{ \abs{\ip{B}{X}} \ge t }
  \le C_d \exp \left( -c_d \cdot \min_{1 \le k \le d} \; \min_{I_1 \sqcup \cdots \sqcup I_k = [d]}
    \left( \frac{t}{\norm{B}_{I_1,\ldots,I_k}} \right)^{2/k} \right).
  \end{equation}
  The second minimum in the right hand side is over all partitions of the index set $[d] = \{1,\ldots,d\}$ 
  into $k$ subsets, and $\norm{B}_{I_1,\ldots,I_k}$ are certain norms of $B$, which interpolate
  between the Hilbert-Schmidt and the operator norms; 
  see \cite{Latala} for details.
       
  Lehec \cite{Lehec} gave an alternative, shorter proof of \eqref{eq: Latala} based on 
  Talagrand's majorizing measure theorem. 
  Adamczak and Latala \cite{AL} extended \eqref{eq: Latala} for log-concave distributions. 
  Adamczak and Wolff \cite{AW} considered distributions that satisfy a Sobolev inequality, 
  and he gave an extension of \eqref{eq: Latala} for general differentiable functions of $X$
  (not just polynomials). 
  G\"otze, Sambale, and Sinulis extended \eqref{eq: Latala} for 
  $\a$-subexponential distributions.
  Very recently, Adamczak, Latala, and Meller \cite{ALM} explored extensions of \eqref{eq: Latala}
  where the coefficients of $B$ are not scalars but vectors in some Banach space. 
  We refer the paper to the papers \cite{AW, ALM} for a review of the vast 
  literature on concentration inequalities for multivariate polynomials, 
  U-statistics, and more general functions of independent random variables. 
  
  As we noted, inequality \eqref{eq: Latala}
  is well suited to study concentration properties of the quantity $\norm{AX}_H^2$
  that appears in Theorem~\ref{thm: Euclidean concentration tensor}. 
  However, Latala's concentration inequality \eqref{eq: Latala}, as well as all its known extensions 
  and ramifications, contain factors $C_d$ and $c_d$ that depend on the degree $d$ 
  of tensors in some unspecified way; the dependence seems to be at least 
  exponential in $d$. 
  In contrast, Theorems~\ref{thm: convex concentration tensor} 
   and \ref{thm: Euclidean concentration tensor}
  feature the optimal dependence on the degree $d$ of tensors: 
  the constant $c$ in both theorems does not depend on $d$ at all. 
  This can be critical in some applications, one of which we discuss next.

\subsection{Application: random tensors are well conditioned}

Our work was primarily inspired by a question that arose recently 
in the theoretical computer science community \cite{Baskara, Anderson, Abbe, Anari}:
{\em Are random tensors well conditioned?}

Suppose $X_1,\ldots,X_m$ are independent copies of a simple random tensor 
$X \coloneqq x_1 \otimes \cdots \otimes x_d$. How large can $m$ be so that these 
random tensors are linearly independent with high probability? Certainly $m$ can not
exceed the dimension $n^d$ of the tensor space, but can it be arbitrarily 
close to the dimension, say $m = 0.99 n^d$?
Moreover, instead of linear independence we may ask for a stronger, more quantitative
property of being well conditioned. We would like to have a uniform bound
\begin{equation}	\label{eq: well conditioned}
\Big\| \sum_{i=1}^m a_i X_i \Big\|_2 \ge \s \norm{a}_2
\quad \text{for all } a = (a_1,\ldots,a_m) \in \R^m
\end{equation}
with $\s$ as large as possible. 
Equivalently, we can understand \eqref{eq: well conditioned} as a lower bound  
on the {\em smallest singular value} of the $n^d \times m$ matrix $\mathbf{X}^{\odot d}$ 
(called the {\em Khatri-Sidak product}) 
whose columns are vectorized tensors $X_1,\ldots,X_m$:
$$
\s_{\min} \big( \mathbf{X}^{\odot d} \big) \ge \s.
$$

Problems of this type were studied recently in the theoretical computer science community
in the context of computing tensor decompositions \cite{Baskara, Anari}, learning Gaussian 
mixtures \cite{Anderson} and estimating the capacity of error correcting codes \cite{Abbe}.

For $d=1$, this problem has been extensively studied in random matrix theory
\cite{KKS, Tao-Vu RSA, Rudelson Annals, Tao-Vu Annals, RV square, Adamczak-et-al, RV rectangular, Tao-Vu GAFA, RV ICM, GNT, Basak-Rudelson, Tikhomirov Advances, Tikhomirov Israel}.
Since $\mathbf{X}^{\odot 1}$ is an $n \times m$ random matrix whose entries are 
independent, mean zero, unit variance, subgaussian random variables, 
known optimal results \cite{RV rectangular} yield the bound
\begin{equation}	\label{eq: invertibility d=1}
\s_{\min} \big( \mathbf{X}^{\odot 1} \big) 
\gtrsim \e \sqrt{n}
\end{equation}
if $m = (1-\e)n$ and $\e \in (0,1]$.

For $d \ge 2$, optimal results on $\s = \s_{\min} \big( \mathbf{X}^{\odot d} \big)$ are yet unknown. 
Various random models were studied: 
the factors $x_k$ of the simple tensor $X$ were assumed to be Gaussian
(possibly with nonzero means) in \cite{Baskara}, Bernoulli in \cite{Abbe} and are allowed 
to have very general general distributions with non-degenerate marginals in \cite{Anari}. 
Symmetric random tensors are considered in \cite{Abbe}.
Baskara et al. \cite{Baskara} obtained the lower bound \eqref{eq: well conditioned} 
with $\s = (1/n)^{\exp(O(d))}$ for $d = o(\log \log n)$; Anari et al. \cite{Anari} 
improves this to $\s = (1/O(n))^d$ for $d = o(\sqrt{n/\log n})$, and 
Abbe et al. \cite{Abbe} guarantees linear independence (i.e. $\s > 0$) for symmetric random 
tensors if $d = o(\sqrt{n/\log n})$. For a related notion of {\em row product} of random matrices, 
the problem was studied by Rudelson \cite{Rudelson}.

In this paper, we prove a bound on the smallest singular value 
$\s_{\min} \big( \mathbf{X}^{\odot d} \big)$ is of constant order. 
We derive it as an application of Theorem~\ref{thm: Euclidean concentration tensor}.
Let us give an informal statement here; Corollary~\ref{cor: well conditioned} will provide
a more rigorous version.

\begin{corollary}[Random tensors are well conditioned]		\label{cor: well conditioned intro}
  If $d = o(\sqrt{n/\log n})$ and $\e \in (0,1)$, then $m = (1-\e) n^d$ independent simple 
  subgaussian random tensors  $X_1,\ldots,X_m$ in $\R^{n^d}$ are well   
  conditioned with high probability:
  $$
  \Big\| \sum_{i=1}^m a_i X_i \Big\|_2 \ge \frac{\sqrt{\e}}{2} \, \norm{a}_2
  \quad \text{for all } a = (a_1,\ldots,a_m) \in \R^m.
  $$
\end{corollary}

\subsection{Our approach}			\label{s: approach}

Let us briefly explain our approach to tensor concentration.
Suppose first that we do not care about the dependence on the degree $d$. 
Then Theorem~\ref{thm: convex concentration tensor}  
can be proved by expressing the deviation $f-\E f$ as a telescopic sum 
and controlling each increment by Talagrand's Theorem~\ref{thm: convex concentration}.
For example, if $d=3$, then for the function $f = f(x \otimes y \otimes z)$ we would write
$$
f - \E f = (f - \E_x f) + (\E_x f - \E_{x,y} f) + (\E_{x,y} f - \E_{x,y,z} f)
\eqqcolon \Delta_1 + \Delta_2 + \Delta_3
$$
where $\E_{x}$ denotes the conditional expectation with respect to $x$ (conditional on $y$ and $z$),
and similarly for $\E_{x,y}$.
Applying Theorem~\ref{thm: convex concentration} for $f$ as a function of $x$, we control $\Delta_1$;
applying the same theorem for $\E_x f$ as a function of $y$, we control $\Delta_2$, 
and applying it again for $\E_{x,y} f$ as a function of $z$, we control $\Delta_3$.
Then we combine all increments $\Delta_k$ by the triangle inequality.

This argument, however, would produce an exponential dependence of $d$ in the 
concentration inequality. This is because the Lipschitz norm of $f$ as a function of $x$ 
is bounded by 
$$
\norm{y}_2 \norm{z}_2 \le K^{d/2}
$$
if all coefficients of $y$ and $z$ are bounded by $K$ a.s.

To get a better control of the Lipschitz norms of all functions that appear. in the telescopic sum, 
we prove a {\em maximal inequality} in Section~\ref{s: maximal inequality}, 
which provides us with a uniform bound on the products of norms of independent random vectors.
This allows us to avoid losing any factors that are exponential in $d$.

However, combining the increments $\Delta_k$ by a simple union bound and triangle inequality
is suboptimal and leads to an extra factor that is linear in $d$. One can avoid this by noting 
that $\Delta_k$ are martingale differences and using martingale concentration techniques
(coupled with the maximal inequality). This is the approach we chose to prove 
Theorem~\ref{thm: convex concentration tensor}.

One can try to prove Theorem~\ref{thm: Euclidean concentration tensor} in a similar way,
but a new difficulty arises here. We may not simply choose 
$f(x \otimes y \otimes z) = \norm{A(x \otimes y \otimes z)}_2$
and write the telescopic sum for it. This is because 
$\E_x f$ is not a Euclidean function in $y$, i.e. 
it can not be expressed as $\norm{By}$ for any linear operator $B$ mapping $\R^n$
into a Hilbert space, so we may not use Theorem~\ref{thm: Euclidean concentration} to control 
the deviation of $\E_x f$.

This forces us to work with $f^2$ instead of $f$, since then $(\E_x f^2)^{1/2}$ is a Euclidean 
function. Thus we write 
$$
f^2 - \E f^2 = (f^2 - \E_x f^2) + (\E_x f^2 - \E_{x,y} f^2) + (\E_{x,y} f^2 - \E_{x,y,z} f^2)
\eqqcolon \Delta_1 + \Delta_2 + \Delta_3.
$$
Squaring $f$, however, produces tails of the increments that are heavier than subgaussian. 
This prompts us to abandon the use of Theorem~\ref{thm: Euclidean concentration}. 
Instead of controlling the tails of the increments, we control their moment generating function (MGF).
In the end, we still combine the MGF's of the increments using a martingale-like argument
coupled with the maximal inequality. This ultimately leads to Theorem~\ref{thm: Euclidean concentration tensor}.

\begin{remark}[An alternative approach to convex concentration for random tensors] 
\label{rem: alternative approach}
  There is an alternative and somewhat simpler way to prove 
  Theorem~\ref{thm: convex concentration tensor}, where one won't have to use
  a maximal inequality.
  Instead, one can deduce this result (with some work) from a version of convex concentration 
  (Theorem~\ref{thm: convex concentration}) 
  that holds for a weaker notion of convexity, namely 
  for separately convex functions, i.e. functions that are convex in each coordinate
  \cite{Talagrand, Maurer}, see \cite[Theorems~6.10, 6.9]{BLM}.
  However, there seem to be no simpler way to prove Theorem~\ref{thm: Euclidean concentration tensor}, a result we care most about in view of applications. 
  So, for pedagogical reasons we choose to prove 
  Theorem~\ref{thm: convex concentration tensor} using maximal inequality, 
  so we can use it later as a stepping stone for the proof of 
  Theorem~\ref{thm: Euclidean concentration tensor}.
\end{remark}

\begin{remark}[A broader view]
  The method we develop here is flexible and might be used to ``tensorize'' some other  
  concentration inequalities. 
  For example, if all $x_k$ have the standard normal distribution, 
  then the convexity requirement is not needed in Theorem~\ref{thm: convex concentration tensor},
  and we get a tensor version of the Gaussian concentration inequality \eqref{eq: conc Gaussian}.
  Furthermore, one should be able to relax the subgaussian assumption in 
  Theorem~\ref{thm: Euclidean concentration tensor} 
  by using in our argument a version Theorem~\ref{thm: Euclidean concentration} 
  for heavier tails 
  obtained recently by G\"otze, Sambale, and Sinulis \cite{GSS}.
\end{remark}

\subsection{Open problems}	

We do not know optimal concentration inequalities for {\em symmetric} tensors
$X = x^{\otimes d} = x \otimes \cdots \otimes x$. One could possibly use decoupling 
to reduce the problem to concentration to tensors with independent factors, and then 
apply Theorem~\ref{thm: convex concentration tensor} or \ref{thm: Euclidean concentration tensor}. 
However, decoupling will likely cause a loss of factors that are exponential in $d$, 
which will defeat our purpose. 

There are many directions in which Corollary~\ref{cor: well conditioned intro} should be 
strengthened and generalized. Can the bound be improved to $\e n^{d/2}$, matching 
the inequality \eqref{eq: invertibility d=1} for $d=1$? Does it hold for degrees higher 
than $\sqrt{n}$, for example for $d \asymp n$? Even linear independence is unknown for higher
degrees. Can Corollary~\ref{cor: well conditioned intro} be extended to other models 
of randomness considered in the theoretical computer science community 
\cite{Baskara, Anderson, Abbe, Anari}? For example, does it hold for symmetric tensors, 
and can the mean zero and subgaussian assumptions be significantly weakened?

\subsection{The rest of the paper}
In Section~\ref{s: prelims} we collect some basic facts from high-dimensional probability
that will be needed later. Most importantly, in Proposition~\ref{prop: MGF chaos}
we show how to control the MGF of a random chaos
of order $2$, which is the quadratic form $x^\tran M x$ where $x$ is a random vector and
$M$ is a fixed matrix. In Proposition~\ref{prop: MGF chaos} we derive a version of 
Hanson-Wright inequality in the MGF form. These results, although possibly known 
as a folklore, are hard to find in the literature and could be useful for future applications.

In Section~\ref{s: maximal inequality} we prove a sharp {\em maximal inequality} 
for products of norms of independent random vectors.
We use it to establish our main results:
in Section~\ref{s: Talagrand tensor} we prove Theorem~\ref{thm: convex concentration tensor}
and in Section~\ref{s: subgaussian tensor} we prove Theorem~\ref{thm: Euclidean concentration tensor}.

In Section~\ref{s: applications} we give two applications to the geometry of random tensors.
We prove a concentration inequality for the distance between a random tensor and 
a given subspace in Corollary~\ref{cor: dist}, and then use it to show that
random tensors are well conditioned, proving a formal
version of Corollary~\ref{cor: well conditioned intro}.

\subsection*{Acknowledgement}
The author is grateful to Mark~Rudelson for useful comments on the preliminary 
version of this manuscript, in particular for pointing out an alternative 
approach to Theorem~\ref{thm: convex concentration} mentioned in 
Remark~\ref{rem: alternative approach}.

\section{Preliminaries}			\label{s: prelims}

Throughout this paper, we use basic facts about subgaussian and subexponential 
random variables that can be found e.g. in \cite[Chapters~2--3]{V book} and
\cite[Chapter~2]{Wainwright}. 
Positive constants are denoted by 
$C, c, C_1, c_1, \ldots$, and their specific values can be different in different
parts of this paper. 
We allow these constants to depend only on the 
a.s. bound of the coefficients (for Theorem~\ref{thm: convex concentration}) 
or the subgaussian norms of the coefficients (for Theorem~\ref{thm: Euclidean concentration}), 
but not on any other parameters. 

\subsection{Concentration of the norm}
One such fact that follows immediately from Theorem~\ref{thm: Euclidean concentration}
for the identity matrix $A$ is a concentration inequality for the norm of a random vector.

\begin{corollary}[Concentration of norm]		\label{cor: conc norm}
  Let $x$ be a random vector in $\R^n$ whose coordinates are 
  independent, mean zero, unit variance, subgaussian random variables.
  Then, for every $t \ge 0$, we have
  $$
  \Pr{ \abs{\norm{x}_2 - \sqrt{n}} > t } 
  \le 2 \exp(-ct^2).
  $$
  Here $c>0$ depends only on the bound on the subgaussian norms. 
\end{corollary}

\subsection{MGF of a quadratic form}
Another useful fact is the following bound on the moment generating function (MGF)
of a chaos of order $2$. It might be known but is hard to find in the literature
in this generality.

\begin{proposition}[MGF of a subgaussian chaos]			\label{prop: MGF chaos}
  Let $M$ be an $n \times n$ matrix.
  Let $x$ be a random vector in $\R^n$ whose coordinates are 
  independent, mean zero, unit variance, subgaussian random variables.
  Then
  $$
  \E \exp \Big( \l (x^\tran M x - \tr M) \Big)
  \le \exp \left( C \l^2 \norm{M}_\HS^2 \right)
  $$
  for every $\l \in \R$ such that $\abs{\l} \le c/\norm{M}_\op$.
\end{proposition}

\begin{proof}
{\em Step 1. Separating the diagonal and off-diagonal parts.}
We can break the quadratic form as follows:
$$
S \coloneqq x^\tran M x - \tr M
= \sum_{i=1}^n M_{ii} (x_i^2-1) + \sum_{i,j:\; i \ne j} M_{ij} x_i x_j
\eqqcolon S_{\textrm{diag}} + S_{\textrm{offdiag}}.
$$
By Cauchy-Schwarz inequality, we have
$$
\E \exp(\l S) 
\le \Big[ \E \exp \left( 2 \l S_{\textrm{diag}} \right) \Big]^{1/2}
	\Big[ \E \exp \left( 2 \l S_{\textrm{offdiag}} \right) \Big]^{1/2}.
$$
Let us consider the diagonal and off-diagonal parts separately.

\medskip

{\em Step 2. Diagonal part.}
Since $x_i$ are subgaussian random variables with unit variance, 
$x_i^2-1$ are mean-zero, subexponential random variables, and
$$
\norm[1]{x_i^2 - 1}_{\psi_1}
\lesssim \norm[1]{x_i^2}_{\psi_1}
= \norm[1]{x_i}_{\psi_2}^2
\lesssim 1.
$$
(This is a combination of some basic facts about subgaussian and subexponential
distributions, see \cite[Exercise~2.7.10 and Lemma~2.7.6]{V book}.)
Then a standard bound on the MGF of a mean-zero, subexponential 
distribution (\cite[Property~5 in Proposition~2.7.1]{V book}) gives
\begin{equation}	\label{eq: MGF xi2}
\E \exp \left( \l_i (x_i^2-1) \right) \le \exp(C \l_i^2)
\quad \text{if } \abs{\l_i} \le c.
\end{equation}
Therefore
\begin{align*} 
\E \exp(2 \l S_{\textrm{diag}}) 
  &= \prod_{i=1}^n \E \exp \left( 2 \l M_{ii} (x_i^2-1) \right)
  	\quad \text{(by independence)} \\
  &\le \exp \Big( C \l^2 \sum_{i=1}^n M_{ii}^2 \Big)
  	\quad \text{if } \abs{\l} \le \frac{c}{\abs{M_{ii}}} 
	\quad \text{(by \eqref{eq: MGF xi2})} \\
  &\le \exp \Big( C \l^2 \norm{M}_\HS^2 \Big)
    	\quad \text{if } \abs{\l} \le \frac{c}{\norm{M}_\op}.
\end{align*}

\medskip

{\em Step 2. Off-diagonal part.}
Let $x_1', \ldots, x_n'$ be independent copies of $x_1, \ldots, x_n$. We have 
\begin{align*} 
\E \exp(2 \l S_{\textrm{offdiag}}) 
  &= \E \exp \Big( 2 \l \sum_{i,j:\; i \ne j} M_{ij} x_i x_j \Big) \\
  &\le \E \exp \Big( 8 \l \sum_{i,j=1}^n M_{ij} x_i x_j' \Big) 
  	\quad \text{(by decoupling, see \cite[Remark~6.1.3]{V book})}  \\
  &\le \exp \Big( C \l^2 \norm{M}_\HS^2 \Big)
    	\quad \text{if } \abs{\l} \le \frac{c}{\norm{M}_\op},
\end{align*}
where the last bound follows from \cite[Lemmas~6.2.3 and 6.2.2]{V book}.

Combining the diagonal and off-diagonal contributions, we complete the proof.
\end{proof}

\begin{corollary}		\label{cor: MGF Ax}
  Let $H$ be a Hilbert space and 
  $A: (\R^n, \norm{\cdot}_2) \to H$ be a linear operator. 
  Let $x$ be a random vector in $\R^n$ whose coordinates are 
  independent, mean zero, unit variance, subgaussian random variables.
  Then 
  $$
  \E \exp \Big( \l \big( \norm{Ax}_H^2 - \norm{A}_\HS^2 \big) \Big) 
  \le \exp \left( C \l^2 \norm{A}_\op^2 \norm{A}_\HS^2 \right)
  $$
  for every $\l \in \R$ such that $\abs{\l} \le c/\norm{A}_\op^2$.
\end{corollary}

\begin{proof}
Apply Proposition~\ref{prop: MGF chaos} for $M \coloneqq A^*A$ 
and note that 
$$
x^\tran M x = \norm{Ax}_H^2, \quad
\tr M = \norm{A}_\HS^2, \quad
\norm{M}_\op = \norm{A}_\op^2, \quad
\norm{M}_\HS \le \norm{A}_\op \norm{A}_\HS. \qedhere
$$
\end{proof}

Note in passing that Corollary~\ref{cor: MGF Ax} implies Euclidean concentration Theorem~\ref{thm: Euclidean concentration}. All one needs to do is use exponential Markov's inequality and
optimize the resulting bound in $\l$. We leave this as an exercise.

\subsection{Euclidean functions}			\label{s: Euclidean functions}

Theorems~\ref{thm: Euclidean concentration}
and \ref{thm: Euclidean concentration tensor} can be conveniently stated as 
results about concentration of Euclidean functions. 

\begin{definition}[Euclidean functions]
  A function $f: \R^n \to [0,\infty)$ is called a {\em Euclidean} function on $\R^n$
  if it can be expressed as 
  $$
  f(x) = \norm{Ax}_H
  $$
  where $H$ is a Hilbert space and $A: \R^n \to H$ is a linear operator. 
  Equivalently, $f$ is Euclidean if $f^2$ is a positive-semidefinite quadratic form, 
  i.e. if
  $$
  f(x)^2 = x^\tran M x
  $$
  for some $n \times n$ positive-semidefinite matrix $M$. 
\end{definition} 

Let us note a few obvious facts about Euclidean functions:

\begin{lemma}[Properties of Euclidean functions]
  \quad
  \begin{enumerate}[(i)]
    \item If $f$ is a Euclidean function on $\R^n$ then $af$ is, for any $a \ge 0$.
    \item If $f$ and $g$ are Euclidean functions on $\R^n$ then $\sqrt{f^2 + g^2}$ is.
    \item If $f$ is a random Euclidean function on $\R^n$ then $(\E f^2)^{1/2}$ is.
    \item The Lipschitz norm of a Euclidean function can be computed as follows:
      $$
      \norm{f}_\Lip = \max_{x \in \R^n:\; \norm{x}_2 = 1} f(x).
      $$
      In particular, if $f(x) = \norm{Ax}_H$ then $\norm{f}_\Lip = \norm{A}_\op$.
  \end{enumerate}
\end{lemma}

For future convenience, we restate 
Corollary~\ref{cor: MGF Ax} in terms of Euclidean functions:

\begin{corollary}[MGF of a Euclidean function]		\label{cor: MGF Euclidean}
  Let $f : (\R^n, \norm{\cdot}_2) \to [0,\infty)$ be a Euclidean function.  
  Let $x$ be a random vector in $\R^n$ whose coordinates are 
  independent, mean zero, unit variance, subgaussian random variables.
  Then 
  $$
  \E \exp \Big( \l \big( f(x)^2 - \E f(x)^2 \big) \Big) 
  \le \exp \left( C \l^2 \norm{f}_\Lip^2 \E f(x)^2 \right)
  $$
  for every $\l \in \R$ such that $\abs{\l} \le c/\norm{f}_\Lip^2$.
\end{corollary}

\section{A maximal inequality}			\label{s: maximal inequality}

The proof of both of our main results, 
Theorems~\ref{thm: convex concentration tensor} and \ref{thm: Euclidean concentration tensor}, 
relies on a tight control of the norm of the simple random tensor
\begin{equation}	\label{eq: X norm}
\norm{X}_2 = \norm[1]{x_1 \otimes \cdots \otimes x_d}_2
= \prod_{i=1}^d \norm[1]{x_i}_2.
\end{equation}

\begin{lemma}[The norm of a random tensor]		\label{lem: prod}
  Let $x_1, \ldots, x_d \in \R^n$ be independent random vectors
  with independent, mean zero, unit variance, subgaussian coordinates. 
  Then, for every $0 \le t \le 2n^{d/2}$, we have
  $$
  \Pr{ \prod_{i=1}^d \norm[1]{x_i}_2 > n^{d/2} + t}
  \le 2 \exp \Big( - \frac{c t^2}{d n^{d-1}} \Big).
  $$
\end{lemma}

Note in passing that this result is a partial case of 
Theorem~\ref{thm: convex concentration tensor}
for the function $f(X) = \norm{X}_2$ 
and of Theorem~\ref{thm: Euclidean concentration tensor}
for the identity map $A$.

\begin{proof}
Let $s \ge 0$. Then
\begin{equation}	\label{eq: prob prod}
\Pr{ \prod_{i=1}^d \norm[1]{x_i}_2  > (\sqrt{n} + s)^d }
\le \Pr{ \frac{1}{d} \sum_{i=1}^d \big( \norm[1]{x_i}_2 - \sqrt{n} \big) > s }.
\end{equation}
To check this, take $d$-th root on both sides of the inequality
$\prod_{i=1}^d \norm[1]{x_i}_2  > (\sqrt{n} + s)^d$, apply the inequality of 
arithmetic and geometric means, and subtract $\sqrt{n}$ from both parts.
Furthermore, we can replace all terms $\sqrt{n}$ in \eqref{eq: prob prod}
by the means $\mu_i \coloneqq \E \norm[1]{x_i}_2$
since they satisfy
$$
\mu_i \le \Big( \E \norm[1]{x_i}_2^2 \Big)^{1/2} = \sqrt{n}.
$$
Thus the probability in \eqref{eq: prob prod} is bounded by
\begin{equation}	\label{eq: prob prod revised}
\Pr{ \frac{1}{d} \sum_{i=1}^d \big( \norm[1]{x_i}_2 - \mu_i \big) > s },
\end{equation}
which is a tail probability for a sum of independent, mean zero random variables.

By the concentration of the norm (Corollary~\ref{cor: conc norm})
and a standard centering argument (\cite[Lemma~2.6.8]{V book}),
we have\footnote{Here $\norm{\cdot}_{\psi_2}$ and $\norm{\cdot}_{\psi_1}$ denote
  the subgaussian and subexponential norms, respectively; see \cite[Sections~2.5, 2.7]{V book}
  for definition and basic properties.}
$$
\norm{ \norm[1]{x_i}_2 - \mu_i }_{\psi_2} \le C, 
\quad i=1,\ldots,d.
$$ 
and this implies that
$$
\norm[3]{ \frac{1}{d} \sum_{i=1}^d \big( \norm[1]{x_i}_2 - \mu_i \big)}_{\psi_2} 
\le \frac{C}{\sqrt{d}},
$$
see \cite[Proposition~2.6.1]{V book}.
Thus the probability in \eqref{eq: prob prod revised} is bounded by 
$$
2\exp(-c s^2 d).
$$

Let $0 \le u \le 2$ and apply this bound for $s \coloneqq u \sqrt{n}/(2d)$.
With this choice, 
$$
(\sqrt{n} + s)^d = n^{d/2} \Big( 1 + \frac{u}{2d} \Big)^d 
\le n^{d/2} (1+u).
$$
Thus we have shown that 
\begin{equation}	\label{eq: conc prod u}
\Pr{ \prod_{i=1}^d \norm[1]{x_i}_2 > n^{d/2} (1+u) }
\le 2\exp \Big( -\frac{c n u^2}{4d} \Big).
\end{equation}
Using this inequality for $u \coloneqq t/n^{d/2}$, we complete the proof.
\end{proof}

A stronger statement will be needed in the proof of our main results: 
we will require a tight control of the products \eqref{eq: conc prod u} {\em for all $d$
simultaneously}. The following maximal inequality will be used for that: 

\begin{lemma}[A maximal inequality]			\label{lem: maximal inequality}
  Let $x_1, \ldots, x_d \in \R^n$ be independent random vectors
  with independent, mean zero, unit variance, subgaussian coordinates. 
  Then, for every $0 \le u \le 2$, we have
  $$
  \Pr{ \max_{1 \le k \le d} n^{-k/2} \prod_{i=1}^k \norm[1]{x_i}_2  > 1+u}
  \le 2 \exp \Big( - \frac{c n u^2}{d} \Big).
  $$
\end{lemma}

\begin{proof}
{\em Step 1. A binary partition.}
By increasing $d$ if necessary, we can assume that 
$$
d = 2^L \quad \text{for some } L \in \N.
$$

For each level $\ell \in \{0,1,\ldots,L\}$, consider the partition $\II_\ell$ of 
the integer interval $[1,d] = \{1,\ldots,d\}$ into $2^\ell$ successive intervals of length
$$
d_\ell \coloneqq \frac{d}{2^\ell}.
$$
We call each of these intervals a {\em binary interval}. For example,  
the family $\II_0$ consists of just one binary interval $[1,d]$, 
and the family $\II_1$ consists of two binary intervals $[1,d/2]$ and $[d/2+1,d]$. 

For every integer $k \in [1,d]$, the interval $[1,k]$ can be partitioned into binary intervals
of different lengths. (The binary representation of the number $k/d$ determines which intervals
participate in this partition.) As a consequence, such partition of $[1,k]$ must include no more than 
one interval from each family $\II_\ell$.  

\medskip

{\em Step 2. Controlling the product over binary sets.}
Fix a level $\ell \in \{0,1,\ldots,L\}$ and a binary interval $I \in \II_\ell$.
Apply Lemma~\ref{lem: prod} with $d$ replaced by $|I|=d_\ell = d/2^\ell$ and 
for $t \coloneqq 2^{-\ell/4} n^{d_\ell/2} u$. It gives
$$
\Pr{ \prod_{i \in I} \norm[1]{x_i}_2 > (1 + 2^{-\ell/4} u) n^{d_\ell/2} }
\le 2 \exp \Big( - \frac{c_0 n u^2}{2^{\ell/2} d_\ell} \Big)
= 2 \exp \Big( - 2^{\ell/2} \cdot \frac{c_0 n u^2}{d} \Big).
$$

Taking a union bound over all levels $\ell$ and all $2^\ell$ binary intervals $I$ in 
the family $\II_\ell$, we get
\begin{multline*}
\Pr{ \exists \ell \in \{0,\ldots,L\}, \exists I \in \II_\ell :\;   
\prod_{i \in I} \norm[1]{x_i}_2 > (1 + 2^{-\ell/4} u) n^{d_\ell/2} } \\
\le \sum_{\ell=0}^L 2^\ell \cdot 2\exp \Big( - 2^{\ell/2} \cdot \frac{c_0 n u^2}{d} \Big).
\end{multline*}

To simplify this bound, we can assume that $c_0 n u^2/d \ge 1$, otherwise the
probability bound in the conclusion of the lemma becomes trivial if $c<c_0/2$.
Also, $2^{\ell/2} \ge 1$, and thus 
$$
2^{\ell/2} \cdot \frac{c_0 n u^2}{d} \ge \frac{1}{2} \Big( 2^{\ell/2} + \frac{c_0 n u^2}{d} \Big).
$$
Substituting this into our probability bound, we can continue it as
$$
2\exp \Big( - \frac{c_0 n u^2}{2d} \Big).
\sum_{\ell=0}^L 2^\ell \cdot 2\exp \Big( - 2^{\ell/2-1} \Big)
\le C \exp \Big( - \frac{c_0 n u^2}{2d} \Big).
$$
By reducing the absolute constant $c_0$, we can make $C = 2$. 

\medskip

{\em Step 3. Controlling the product over any interval.}
Let us fix a realization of random vectors for which the good event considered above occurs, i.e.
\begin{equation}	\label{eq: prod bin int}
\prod_{i \in I} \norm[1]{x_i}_2 \le (1 + 2^{-\ell/4} u) n^{d_\ell/2}
\quad \text{for every } \ell \in \{0,\ldots,L\} \text{ and } I \in \II_\ell.
\end{equation}
Let $1 \le k \le d$. As we noted in Step~1, we can partition the interval $[1,k]$
into binary intervals $I \in \II_\ell$ so that at most one binary interval is taken from each family $\II_\ell$.
Let us multiply the inequalities \eqref{eq: prod bin int} for all binary intervals $I$ that participate in 
this partition. Note that sum of exponents $d_\ell$ is the sum of the length of these intervals $I$, 
which equals $k$. Thus we obbtain
\begin{align*} 
\prod_{i=1}^k \norm[1]{x_i}_2 
  &\le n^{k/2} \prod_{\ell=0}^L (1 + 2^{-\ell/4} u) 
  \le n^{k/2} \exp \Big( u \sum_{\ell=0}^L 2^{-\ell/4} \Big)
  		\quad \text{(using $1+x \le e^x$)}\\
  &\le n^{k/2} \exp(Cu) 
  \le n^{k/2} (1+e^{2C} u) 
  		\quad \text{(since $0 \le u \le 2$)}.
\end{align*}

This yields the conclusion of the lemma with $Cu$ instead of $u$ in the bound. 
One can get rid of $C$ by reducing the constant $c$ in the probability bound.
The proof is complete.
\end{proof}

\section{Proof of Theorem~\ref{thm: convex concentration tensor}}		\label{s: Talagrand tensor}

{\em Step 1. Applying the maximal inequality.}
We can assume without loss of generality that $\norm{f}_\Lip = 1$.
Consider the events 
$$
\EE_k \coloneqq \left\{ \prod_{i=k}^d \norm[1]{x_i}_2 \le 2n^{(d-k+1)/2} \right\}, 
\quad k=1, \ldots, d,
$$
and let $\EE_{d+1}$ be the entire probability space for convenience.
Applying the maximal inequality of Lemma~\ref{lem: maximal inequality} for $u=1$
and for the reverse ordering of the vectors, we see that the event
$$
\EE \coloneqq \EE_2 \cap \cdots \cap \EE_d
$$
is likely:
\begin{equation}	\label{eq: EE prob}
\P(\EE) \ge 1 - 2 \exp \Big( - \frac{cn}{d} \Big).
\end{equation}

\medskip

{\em Step 2. Applying the convex concentration inequality.}
Fix any realization of the random vectors $x_2, \ldots, x_d$ that satisfy $\EE_2$
and apply the Convex Concentration Theorem~\ref{thm: convex concentration} for 
$f$ as a function of $x_1$. It is a convex and Lipschitz function. 
To get a quantitative bound on its Lipschitz norm, consider any $x, y \in \R^n$ and note that
\begin{align*} 
  &\abs{ f(x \otimes x_2 \otimes \cdots \otimes x_d) 
	- f(y \otimes x_2 \otimes \cdots \otimes x_d)} \\
  &\le \norm{ (x-y) \otimes x_2 \otimes \cdots \otimes x_d }_2
	\quad \text{(since $\norm{f}_\Lip = 1$)} \\
  &= \norm{x-y}_2 \cdot \prod_{i=2}^d \norm[1]{x_i}_2
  \le \norm{x-y}_2 \cdot 2n^{(d-1)/2}
   	\quad \text{(since $\EE_2$ holds)}.
\end{align*}
This shows that $f$ as a function of $x_1$ has Lipschitz norm bounded by 
\begin{equation}	\label{eq: L}
L \coloneqq 2n^{(d-1)/2}.
\end{equation}
The Convex Concentration Theorem~\ref{thm: convex concentration} then yields
$$
\norm{f - \E_{x_1} f}_{\psi_2(x_1)} \le CL
\quad \text{for any } x_2, \ldots, x_d \text{ that satisfy } \EE_2.
$$
In this inequality, $x_1$ indicates that the expectation and the $\psi_2$ norm 
is taken with respect to the random vector $x_1$, i.e. 
conditioned on all other random vectors.

Fix any realization of the random vectors $x_3, \ldots, x_d$ that satisfy $\EE_3$
and apply the Convex Concentration Theorem~\ref{thm: convex concentration} for 
$\E_{x_1} f$ as a function of $x_2$. It is a convex and Lipschitz function. 
To get a quantitative bound on its Lipschitz norm, consider any $x, y \in \R^n$ and note that
\begin{align*} 
  &\abs{ \E_{x_1} f(x_1 \otimes x \otimes x_3 \otimes \cdots \otimes x_d) 
	- \E_{x_1} f(x_1 \otimes y \otimes x_3 \otimes \cdots \otimes x_d) } \\
  &\le \E_{x_1} \norm{ x_1 \otimes (x-y) \otimes x_3 \otimes \cdots \otimes x_d }_2
	\quad \text{(by Jensen's inequality and since $\norm{f}_\Lip = 1$)} \\
  &\le \Big( \E_{x_1} \norm[1]{x_1}_2^2 \Big)^{1/2} \cdot \norm{x-y}_2 \cdot \prod_{i=3}^d \norm[1]{x_i}_2 \\
  &\le \sqrt{n} \cdot \norm{x-y}_2 \cdot 2n^{(d-2)/2}
   	\quad \text{(since $\EE_3$ holds)} \\
  &= \norm{x-y}_2 \cdot 2n^{(d-1)/2}.
\end{align*}
This shows that $\E_{x_1} f$ as a function of $x_2$ has Lipschitz norm bounded by 
$L = 2n^{(d-1)/2}$. 
The Convex Concentration Theorem~\ref{thm: convex concentration} then yields
$$
\norm{\E_{x_1} f - \E_{x_1, x_2} f}_{\psi_2(x_2)} \le CL
\quad \text{for any } x_3, \ldots, x_d \text{ that satisfy } \EE_3.
$$

Continuing in a similar way, we can show that for every $k = 1,\ldots,d$: 
\begin{equation}	\label{eq: differences subgaussian}
\norm{\E_{x_1, \ldots, x_{k-1}} f - \E_{x_1, \ldots, x_k} f}_{\psi_2(x_k)} \le CL
\quad \text{for any } x_{k+1}, \ldots, x_d \text{ that satisfy } \EE_{k+1}.
\end{equation}

\medskip

{\em Step 3. Combining the increments using a martingale-like argument.}
Let us look at the differences
$$
\Delta_k = \Delta_k \big( x_k, \ldots, x_d \big) 
\coloneqq \E_{x_1, \ldots, x_{k-1}} f - \E_{x_1, \ldots, x_k} f.
$$
The estimate \eqref{eq: differences subgaussian} on the subgaussian norm 
yields the following bound on the moment generating function
\cite[Proposition~2.5.2]{V book}:
$$
\E_{x_k} \exp(\l \Delta_k) \le \exp(C L^2 \l^2)
\quad \text{for any } x_{k+1}, \ldots, x_d \text{ that satisfy } \EE_{k+1}
$$
and for any $\l \in \R$. 
We can combine these pieces using a martingale-like argument, which we defer
to Lemma~\ref{lem: martingale}.
It gives that for any $\l \in \R$, 
\begin{equation}	\label{eq: MGF increment}
\E \exp \left( \l (f-\E f) \right) \one_\EE 
= \E \exp \left( \l (\Delta_1+\cdots+\Delta_d) \right) \one_\EE 
\le \exp(C d L^2 \l^2)
\end{equation}
where $\EE = \EE_2 \cap \cdots \cap \EE_d$ is the event whose 
probability we estimated in \eqref{eq: EE prob}.

\medskip

{\em Step 4. Deriving the concentration via exponential Markov's inequality.}
To derive a probability bound, we can use a standard argument 
based on exponential Markov's inequality. Namely, we have for every $\l>0$: 
\begin{align*} 
\Pr{f-\E f > t}
  &\le \Pr{f-\E f > t \text{ and } \EE} + \P(\EE^c) \\
  &= \Pr{\exp(\l(f-\E f)) \one_\EE > \exp(\l t)} + \P(\EE^c) \\
  &\le \exp(-\l t) \, \E \exp \left( \l (f-\E f) \right) \one_\EE + \P(\EE^c) 
  	\quad \text{(by Markov's inequality)} \\
  &\le \exp \left( -\l t + C d L^2 \l^2 \right) + 2 \exp \Big( - \frac{cn}{d} \Big)
  	\quad \text{(by \eqref{eq: MGF increment} and \eqref{eq: EE prob})}.
\end{align*}
This bound is minimized for $\l \coloneqq t/(2CdL^2)$.
With this choice of $\l$, and with the choice of $L$ made in \eqref{eq: L}, 
our bound becomes 
\begin{equation}	\label{eq: prob two terms}
\Pr{f-\E f > t} \le \exp \Big( -\frac{t^2}{16C d n^{d-1}} \Big) + 2 \exp \Big( - \frac{cn}{d} \Big).
\end{equation}

Since by assumption of the theorem, 
\begin{align*} 
t^2 \le 4\E \abs[1]{f \big( x_1 \otimes \cdots \otimes x_d \big) \red{-f(0)} }^2
  &\le 4\E \norm[1]{x_1 \otimes \cdots \otimes x_d}_2^2
  	\quad \text{(since $\norm{f}_\Lip = 1$)}	\\
  &= 4\E \norm[1]{x_1}_2^2 \cdots \norm[1]{x_d}_2^2 
  = 4n^d,
\end{align*}
we have \label{page: using f assumption} 
$$
\frac{t^2}{d n^{d-1}} \le \frac{4n}{d},
$$
This implies that the first term in the bound \eqref{eq: prob two terms} 
dominates over the second, if the constant $C$ is sufficiently large compared to $1/c$.
This gives 
$$
\Pr{f-\E f > t}
\le 3\exp \Big( -\frac{t^2}{16C d n^{d-1}} \Big).
$$

Finally, repeating the argument for $-f$ instead of $f$ we obtain the same probability bound 
for $\Pr{-f+\E f > t}$. Combining the two bounds, we get 
$$
\Pr{\abs{f-\E f} > t}
\le 6\exp \Big( -\frac{t^2}{16C d n^{d-1}} \Big).
$$
We can replace the factor $6$ by $2$ by making $C$ larger if necessary. 
Theorem~\ref{thm: convex concentration tensor} is proved.
\qed

Our argument above used on the following martingale-like inequality, 
which we will carefully state and prove now.

\begin{lemma}[A martingale-type inequality]		\label{lem: martingale}
  Let $x_1,\ldots, x_d$ be independent random vectors. 
  For each $k=1,\ldots,d$, let $f_k = f_k(x_k, \ldots, x_d)$ be an integrable real-valued function
  and $\EE_k$ be an event that is uniquely determined by the vectors $x_{k+1}, \ldots, x_d$.
  Let $\EE_{d+1}$ be the entire probability space for convenience.
  Suppose that, for every $k=1,\ldots,d$:
  $$
  \E_{x_k} \exp(f_k) \le \pi_k
  \quad \text{for every choice of } x_{k+1}, \ldots, x_d \text{ satisfying } \EE_{k+1}.
  $$
  Then, for $\EE \coloneqq \EE_2 \cap \cdots \cap \EE_d$, we have
  $$
  \E \exp(f_1+\cdots+f_d) \one_\EE \le \pi_1 \cdots \pi_d.
  $$
\end{lemma}

\begin{proof}
We have 
\begin{align*} 
\E \exp(f_1+\cdots+f_d) \one_{\EE_2 \cap \cdots \cap \EE_d}
  &= \E \exp(f_2+\cdots+f_d) \one_{\EE_3 \cap \cdots \cap \EE_d}
	\E_{x_1} \exp(f_1) \one_{\EE_2} \\
  &\le \pi_1 \E \exp(f_2+\cdots+f_d) \one_{\EE_3 \cap \cdots \cap \EE_d}
\end{align*}
since $\E_{x_1} \exp(f_1) \one_{\EE_2} \le \pi_1$ a.s. by assumption.
Iterating this argument, we complete the proof.
\end{proof}

\section{Proof of Theorem~\ref{thm: Euclidean concentration tensor}}		\label{s: subgaussian tensor}

Let us restate Theorem~\ref{thm: Euclidean concentration tensor} in terms 
of Euclidean functions which were introduced in Section~\ref{s: Euclidean functions}.

\begin{theorem}[Euclidean concentration for random tensors]	\label{thm: Euclidean concentration tensor restated}
  Let $n$ and $d$ be positive integers and 
  $f: (\R^{n^d}, \norm{\cdot}_2) \to [0,\infty)$ be a Euclidean function. 
  Consider a simple random tensor $X \coloneqq x_1 \otimes \cdots \otimes x_d$ 
  in $\R^{n^d}$, where all $x_k$ are independent random vectors in $\R^n$ whose coordinates
  are independent, mean zero, unit variance, subgaussian random variables.
  Then, for every $0 \le t \le 2(\E f(X)^2)^{1/2}$, we have
  $$
  \Pr{ \abs{f(X) - (\E f(X)^2)^{1/2} } \ge t  } 
  \le 2 \exp \Big( -\frac{ct^2}{d n^{d-1} \norm{f}_\Lip^2} \Big).
  $$
  Here $c>0$ depends only on the bound on the subgaussian norms. 
\end{theorem}

\begin{proof}
{\em Step 1. Applying the maximal inequality.}
The proof starts as in the proof of Theorem~\ref{thm: convex concentration tensor} in Section~\ref{s: Talagrand tensor}. We define the norm-controlling events $\EE_k$ and estimate the probability of 
$\EE = \EE_2 \cap \cdots \EE_d$ by a maximal inequality in the same way as before.

\medskip

{\em Step 2. Applying a subgaussian concentration inequality.}
Fix any realization of the random vectors $x_2, \ldots, x_d$ that satisfy $\EE_2$
and apply Corollary~\ref{cor: MGF Euclidean} for 
$f$ as a function of $x_1$. It is a Euclidean function, and one can check as before that its 
Lipschitz norm is bounded by 
\begin{equation}	\label{eq: L subgaussian}
L \coloneqq 2n^{(d-1)/2}.
\end{equation}
Corollary~\ref{cor: MGF Euclidean} then yields
$$
\E_{x_1} \exp \Big( \l \big( f^2 - \E_{x_1} f^2 \big) \Big) 
\le \exp \left( C \l^2 L^2 \E_{x_1} f^2 \right)
$$
provided that $\abs{\l} \le c/L^2$.
For future convenience, let us restate this bound as follows. 
Choose $\l \in \R$ and denote
$$
\l_0 \coloneqq \l; \quad
\l_1 \coloneqq \l_0 + C \l_0^2 L^2.
$$
Then we have
$$
\E_{x_1} \exp \left( \l_0 f^2 - \l_1 \E_{x_1} f^2 \right) \le 1
\quad \text{for any } x_2, \ldots, x_d \text{ that satisfy } \EE_2,
$$
provided that $\abs{\l_0} \le c/L^2$.

Fix any realization of the random vectors $x_3, \ldots, x_d$ that satisfy $\EE_3$
and apply Corollary~\ref{cor: MGF Euclidean} for 
$(\E_{x_1} f^2)^{1/2}$ as a function of $x_2$. 
It is a Euclidean function whose Lipschitz norm is bounded by $L$ as before. 
Corollary~\ref{cor: MGF Euclidean} then yields
$$
\E_{x_2} \exp \Big( \l_1 \big( \E_{x_1} f^2 - \E_{x_1, x_2} f^2 \big) \Big) 
\le \exp \left( C \l_1^2 L^2 \E_{x_1, x_2} f^2 \right)
$$
provided that $\abs{\l_1} \le c/L^2$.
We can restate this bound as follows. Denote
$$
\l_2 \coloneqq \l_1 + C \l_1^2 L^2.
$$
Then we have
$$
\E_{x_2} \exp \left( \l_1 \E_{x_1} f^2 - \l_2 \E_{x_1, x_2} f^2 \right) \le 1
\quad \text{for any } x_3, \ldots, x_d \text{ that satisfy } \EE_3,
$$
provided that $\abs{\l_1} \le c/L^2$.

Continuing in a similar way, we can show the following for every $k = 1,\ldots,d$.
Denote 
$$
\l_k \coloneqq \l_{k-1} + C \l_{k-1}^2 L^2.
$$
Then we have
$$
\E_{x_2} \exp \left( \l_{k-1} \E_{x_1, \ldots, x_{k-1}} f^2 
	- \l_k \E_{x_1, \ldots, x_k} f^2 \right) 
\le 1
\quad \forall x_{k+1}, \ldots, x_d \text{ that satisfy } \EE_{k+1},
$$
provided that $\abs{\l_{k-1}} \le c/L^2$.

\medskip

{\em Step 3. Combining the increments using a martingale-like argument.}
Combining the pieces into a telescoping sum using Lemma~\ref{lem: martingale}, 
we obtain
\begin{equation}	\label{eq: l0 ld}
\E \exp \left( \l_0 f^2 - \l_d \E f^2 \right) \one_\EE
\le 1
\end{equation}
provided that
\begin{equation}	\label{eq: all lk}
\abs{\l_k} \le \frac{c}{L^2}
\quad \text{for all } k=0,\ldots,d-1.
\end{equation}
If we choose $\l = \l_0 \in \R$ so that $\abs{\l} \le c_0/(dL^2)$ with a sufficiently 
small absolute constant $c_0$, then we can show by induction that 
\eqref{eq: all lk} holds and, moreover,
$$
\l_d \le \l + 2Cd L^2 \l^2.
$$
We defer the verification of both of these bounds to Lemma~\ref{lem: multipliers} below. 
Substituting them into \eqref{eq: l0 ld} and rearranging the terms, we conclude that
$$
\E \exp \left( \l (f^2 - \E f^2) \right) \one_\EE
\le \exp \left( 2Cd L^2 \l^2 \E f^2 \right)
\quad \text{if } \abs{\l} \le \frac{c_0}{dL^2}.
$$
Replacing $\l$ with $-\l$, we see that the same bound holds for 
$\E \exp \left( -\l f^2 + \l \E f^2 \right)$. 
Since the inequality $e^{\abs{z}} \le e^z + e^{-z}$ holds for all $z \in \R$, 
we obtain
$$
\E \exp \left( \l \abs[1]{f^2 - \E f^2} \right) \one_\EE
\le 2\exp \left( 2Cd L^2 \l^2 \E f^2 \right)
\quad \text{if } \abs{\l} \le \frac{c_0}{dL^2}.
$$

\medskip

{\em Step 4. Deriving a concentration inequality for $f^2$.}
Using the exponential Markov's inequality just like 
we did in the proof of Theorem~\ref{thm: convex concentration tensor} in Section~\ref{s: Talagrand tensor}, 
we get
$$
\Pr{\abs[1]{f^2-\E f^2} > u}
\le 2\exp \left( -\l u + 2C d L^2 \l^2 \E f^2 \right) + 2\exp \Big( - \frac{cn}{d} \Big)
$$
for any $u > 0$ and any $0 \le \l \le c_0/(dL^2)$.

Let us optimize the right hand side in $\l$. A good choice is 
$$
\l \coloneqq \frac{c_1}{dL^2} \min \Big( \frac{u}{\E f^2}, \, 1 \Big)
$$
for a sufficiently small constant $c>0$. Indeed, if $c_1 \le c_0$ 
then $\l$ lies in the required range $0 \le \l \le c_0/(dL^2)$, and if $c_1 \le 1/(4C)$ then 
substituting this choice of $\l$ into our probability bound gives
$$
\Pr{\abs[1]{f^2-\E f^2} > u}
\le 2\exp \left( -\frac{c_1}{2dL^2} \min \Big( \frac{u^2}{\E f^2}, \, u \Big) \right) 
	+ 2\exp \Big( - \frac{cn}{d} \Big).
$$

\medskip

{\em Step 5. Deriving a concentration inequality for $f$.}
Choose any $\e \ge 0$ and substitute $u \coloneqq \e \E f^2$ into our probability bound. 
We get
$$
\Pr{\abs[1]{f^2-\E f^2} > \e \E f^2}
\le 2\exp \left( -\frac{c_1}{2dL^2} \min(\e^2,\e) \E f^2 \right) 
	+ 2\exp \Big( - \frac{cn}{d} \Big).
$$

Now choose any $\d \ge 0$ and apply this bound for $\e \coloneqq \max(\d,\d^2)$.
Then $\min(\e^2,\e) = \d^2$, and one can easily check the following implication
$$
\abs[1]{f - (\E f^2)^{1/2}} > \d (\E f^2)^{1/2}
\quad \Longrightarrow \quad
\abs[1]{f^2-\E f^2} > \e \E f^2.
$$
(This follows from the implication 
$\abs{z-1} \ge \d \Rightarrow \abs{z^2-1} \ge \max(\d,\d^2)$
that is valid for all $z \ge 0.)$
Hence we obtain 
$$
\Pr{\abs[1]{f - (\E f^2)^{1/2}} > \d (\E f^2)^{1/2}}
\le 2\exp \left( -\frac{c_1 \d^2 \E f^2}{2dL^2} \right) 
	+ 2\exp \Big( - \frac{cn}{d} \Big).
$$

Now choose any $t \ge 0$ and apply this bound for $\d \coloneqq t/(\E f^2)^{1/2}$.
Recalling the value of $L$ from \eqref{eq: L subgaussian}, we get
$$
\Pr{\abs[1]{f - (\E f^2)^{1/2}} > t}
\le 2\exp \left( -\frac{c_1 t^2}{8 d n^{d-1}} \right) 
	+ 2\exp \Big( - \frac{cn}{d} \Big).
$$

Finally, we can use the theorem's assumption on $t$ to get rid of the second exponential term
just like we did in in the proof of Theorem~\ref{thm: convex concentration tensor} in Section~\ref{s: Talagrand tensor}.
The proof is complete. 
\end{proof}

In Step~3 of the argument above, we used the following bound on the multipliers $\l_k$, which we 
promised to prove later. Let us do it now.

\begin{lemma}[Multipliers]		\label{lem: multipliers}
  Let $d, M \ge 0$ and consider a number $\l_0 \in \R$ such that 
  \begin{equation}	\label{eq: l0}
  \abs{\l_0} \le \frac{1}{8dM}.
  \end{equation}
  Define $\l_1,\ldots,\l_d \in \R$ inductively by the formula
  $$
  \l_k \coloneqq \l_{k-1} + M \l_{k-1}^2, 
  \quad k=1,\ldots,d.
  $$
  Then, for every $k=1,\ldots,d$, we have:
  $$
  \abs{\lambda_k} \le \frac{1}{6dM}
  \quad \text{and} \quad
  \l_k \le \l_0 + 2kM \l_0^2.
  $$
\end{lemma}

\begin{proof}
We can prove the second inequality in the conclusion by induction.
Assume that it holds for some $k$, i.e. 
\begin{equation}	\label{eq: lk vs l0}
\l_k \le \l_0 + 2kM \l_0^2.
\end{equation}
By construction, the sequence $(\l_k)$ is increasing, so the triangle inequality gives
\begin{equation}	\label{eq: abs lk prelim}
\abs{\l_k} 
\le \abs{\l_0} + \abs{\l_k-\l_0}
\le \abs{\l_0} + \l_k-\l_0
\le \abs{\l_0} + 2kM \l_0^2.
\end{equation}
Furthermore, the assumption \eqref{eq: l0} implies that
$$
2kM \l_0^2 \le 2dM \l_0^2 \le \frac{\abs{\l_0}}{4}.
$$
Substituting this into \eqref{eq: abs lk prelim}, we get
\begin{equation}	\label{eq: abs lk vs l0}
\abs{\l_k} \le \frac{5}{4} \abs{\l_0}.
\end{equation}
Then we have
\begin{align*} 
\l_{k+1} 
  &= \l_k + M \l_k^2
  	\quad \text{(by construction)} \\
  &\le \l_0 + 2kM \l_0^2 + M \Big( \frac{5}{4} \abs{\l_0} \Big)^2
   	\quad \text{(by \eqref{eq: lk vs l0} and \eqref{eq: abs lk vs l0})} \\
  &\le \l_0 + 2(k+1) M \l_0^2.
\end{align*}
Thus we proved \eqref{eq: lk vs l0} for $k+1$, so the second inequality in 
the conclusion is verified.

The first bound in the conclusion follows from the first. Indeed, 
using \eqref{eq: abs lk vs l0} and \eqref{eq: abs lk prelim}, we get
$$
\abs{\l_k} 
\le \frac{5}{4} \abs{\l_0}
\le \frac{5}{4} \cdot \frac{1}{8dM} \le \frac{1}{6dM}
$$
as claimed. The proof is complete.
\end{proof}

\section{Applications}				\label{s: applications}

In this section we state and prove a full version of Theorem~\ref{cor: well conditioned intro}
that states that random tensors are well conditioned. But before we do so, let us prove a result that may have an independent interest, namely a concentration inequality for the distance
between a random tensor $X$ and a given subspace $L$. 

\begin{corollary}[Distance to a subspace]		\label{cor: dist}
  Let $n$ and $d$ be positive integers and $L \subset \R^{n^d}$ be a linear subspace
  with $k \coloneqq \codim(L)$.
  Consider a simple random tensor $X \coloneqq x_1 \otimes \cdots \otimes x_d$ 
  in $\R^{n^d}$, where all $x_k$ are independent random vectors in $\R^n$ whose coordinates
  are independent, mean zero, unit variance, subgaussian random variables.
  Then, for every $0 \le t \le 2 \sqrt{k}$, we have
  $$
  \Pr{ \abs{\dist(X,L) - \sqrt{k}} \ge t } 
  \le 2 \exp \Big( -\frac{ct^2}{d n^{d-1}} \Big).
  $$
  Here $c>0$ depends only on the bound on the subgaussian norms.  
\end{corollary}

\begin{proof}
Apply Theorem~\ref{thm: Euclidean concentration tensor} for the orthogonal projection $P$ in $\R^{n^d}$
onto $L^\perp$ and note that $\norm{P}_\op = 1$ and 
$\norm{P}_\HS = \sqrt{\dim(L^\perp)} = \sqrt{k}$.
\end{proof}

For $d=1$, Corollary~\ref{cor: dist} recovers the known optimal concentration inequalities for 
the distance between a random vector and a fixed subspace
are known (see e.g. \cite[Corollary 2.1.19]{Tao book}, \cite{RV}, \cite[Exercise~6.3.4]{V book}), 
which are frequently used in random matrix theory.
Some previously known extensions 
for tensors of degrees $d \ge 2$ were given in \cite{Baskara, Anderson, Abbe, Anari}. 

\medskip

Now we are ready to state and prove a rigorous version of Theorem~\ref{cor: well conditioned intro}:

\begin{corollary}[Random tensors are well conditioned]	\label{cor: well conditioned}
  Consider independent simple subgaussian random tensors $X_1,\ldots,X_m$ 
  (defined like a tensor $X$ in Corollary~\ref{cor: dist}).
  Let $\e$ be such that $Cd^2 \log(n)/n \le \e \le 1/2$.
  If $m \le (1-\e) n^d$ then, with probability at least $1-2\exp(-c \e n/d)$, we have
  $$
  \Big\| \sum_{i=1}^m a_i X_i \Big\|_2 \ge \frac{\sqrt{\e}}{2} \, \norm{a}_2
  \quad \text{for all } a = (a_1,\ldots,a_m) \in \R^m.
  $$
\end{corollary}

Note that this result is nontrivial for $d = O(\sqrt{n/\log n})$, because only in this range
is the range of $\e$ nonempty.

\begin{proof}
We can assume that $\norm{a}_2 = 1$ without loss of generality. 
A simple ``leave-one-out'' bound gives 
\begin{equation}	\label{eq: leave-one-out}
\Big\| \sum_{i=1}^m a_i X_i \Big\|_2 
\ge \frac{1}{\sqrt{m}} \, \min_{j=1,\ldots,m} \dist(X_j, L_j)
\end{equation}
where $L_j$ is the linear span of the $m-1$ vectors $(X_i)_{i \ne j}$. 
Since $\dim(L_j) \le m-1 \le (1-\e) n^d$, we have $\codim(L_j) \ge \e n^d$.

Fix $j$ and apply Corollary~\ref{cor: dist} with $t = \sqrt{\e n^d}/2$ 
conditionally on $(X_i)_{i \ne j}$. It gives
\begin{equation}	\label{eq: dist j}
d(X_j, L_j) \ge \frac{\e n^d}{2}
\end{equation}
with probability at least $1 - 2\exp(-c\e n/2d)$. Taking the union bound over 
$j = 1,\ldots,m$, we conclude that all events \eqref{eq: dist j} hold smultaneously
with probability at least 
$$
1 - 2m\exp \Big( -\frac{c\e n}{2d} \Big)
\ge 1-2\exp \Big( -\frac{c\e n}{4d} \Big),
$$
where we used that $m \le n^d$ and the assumption on $\e$.
Substitute this into the leave-one-out bound \eqref{eq: leave-one-out} 
to complete the proof.
\end{proof}

\end{document}